\documentclass[12pt,a4paper]{article}
\usepackage[utf8]{inputenc}
\usepackage{amsmath}
\usepackage{amsfonts}
\usepackage{amssymb}
\usepackage{amsthm}
\usepackage{enumerate}
\usepackage{fullpage}
\usepackage[colorlinks=true,linkcolor=black,urlcolor=blue,citecolor=black]{hyperref}

\newtheorem{lemma}{Lemma}[section]
\newtheorem{proposition}[lemma]{Proposition}
\newtheorem{theorem}[lemma]{Theorem}
\newtheorem{fact}[lemma]{Fact}
\newtheorem{corollary}[lemma]{Corollary}
\theoremstyle{remark}
\newtheorem{remark}[lemma]{Remark}
\newtheorem{example}[lemma]{Example}

\DeclareMathOperator*{\dom}{dom}
\DeclareMathOperator*{\diag}{diag}
\DeclareMathOperator*{\rank}{rank}

\DeclareMathOperator*{\range}{range}

\DeclareMathOperator*{\vspan}{span}
\DeclareMathOperator*{\Fix}{Fix}
\DeclareMathOperator*{\tr}{tr}

\DeclareMathOperator*{\argmin}{\arg\,\min}

\DeclareMathOperator*{\Ind}{\mathbb{I}}
\newcommand{\trans}{\top}
\newcommand{\E}{\mathbb{E}}
\newcommand{\Sym}{\mathbb{S}}
\newcommand{\Orth}{\mathbb{O}}
\newcommand{\Perm}{\mathbb{P}}
\newcommand{\setto}{\rightrightarrows}
\newcommand{\qede}{\hfill$\Diamond$}

\title{Regularity Properties of Non-Negative Sparsity Sets}
\author{Matthew K. Tam\thanks{Institut f\"ur Numerische und Angewandte Mathematik,
		                     Universit\"at G\"ottingen, 37083 G\"ottingen, Germany. E-mail:~\href{mailto:m.tam@math.uni-goettingen.de}{m.tam@math.uni-goettingen.de}}}

\begin{document}

\maketitle

\begin{abstract}
 This paper investigates regularity properties of two non-negative sparsity sets: non-negative sparse vectors, and low-rank positive semi-definite matrices. Novel formulae for their Mordukhovich normal cones are given and used to formulate sufficient conditions for non-convex notions of regularity to hold. Our results provide a useful tool for justifying the application of projection methods to certain rank constrained feasibility problems.
\end{abstract}

\section{Introduction}
 The solutions of many optimization and reconstruction problems admit characterizations in terms of certain sparse objects. For example, it is sometimes possible to uniquely solve  under-determined linear systems under addition assumptions of sparsity \cite{donoho2006l1sparse}. The difficulty arising in such formulations is in dealing with poorly behaved \emph{sparsity functionals}. Two important examples of such functionals are the $\ell_0$-``norm" for vectors, and the rank function for matrices. It is well known that sparsity functionals lead to problems involving non-convexity and NP-hard complexity (see, for example, \cite{foucart2013compressedSensing,natarajan1995sparse}).
 
 A popular approach to addressing the aforementioned difficulty is to employ \emph{convex relaxations} \cite{l0entropy,stableSignalRecovery,justRelax}, thus allowing for application of industrial strength non-linear solvers. For instance, the $\ell_1$-norm promotes sparsity and has consequently been used as a surrogate for its $\ell_0$ counterpart. Such relaxations come with varying strengths and theoretical guarantees. For an introduction to the topic, we refer the reader to \cite[Ch.~4]{foucart2013compressedSensing}. Whilst one may be able to exactly solve a relaxation, it is not always the case that this translates into a satisfactory sparse solution of the original problem.
 
 An alternative approach involves attempting to deal with the original problem's non-convexity directly \cite{ABT2014matrix,restrictedNormalConesAffine,HLN2014sparse}, and thus avoiding the potential complication of recovering a sparse solution from a convex relaxation. Here one can typically only give theoretical guarantees which apply locally ({\em i.e.,} within some neighbourhood of a solution). In these cases, \emph{regularity properties} of the constraint sets play an important role and can often be usefully formulated in the language of \emph{normal cones}.
  
\bigskip  
  
 This paper investigates regularity properties of sparsity sets having additional non-negativity constraints. We focus on two such sets: non-negative sparse vectors, and low-rank positive semi-definite matrices. Simple, novel formulae for their Mordukhovich normal cones are given, and then used to formulate sufficient conditions to ensure various regularity properties hold. Implications for algorithms and applications are discussed, with particular attention given to \emph{low-rank Euclidean distance matrix reconstruction}.

\bigskip

 The remainder of this paper is organized as follows. In Section~\ref{sec:prelim}, we introduce notation and recall results which will be of use. In Section~\ref{sec:normalCones}, we consider the non-negative sparse vector settings, before \emph{lifting} results to their positive semi-definite counterparts. In Section~\ref{sec:regular}, we deduce consequences of the results from the previous two sections including regularity properties for problems having non-negative sparsity sets. Finally, in Section~\ref{sec:application}, various example applications of problems in which non-negative sparsity sets arise are given.

\section{Preliminaries and Notation}\label{sec:prelim}
Let $\E$ denote a finite dimensional real Hilbert space with inner product $\langle\cdot,\cdot\rangle$ and induced norm $\|\cdot\|$. Throughout this paper we focus on two such spaces. The first is $\mathbb{R}^m$ equipped with the standard inner product. The second is the set of real symmetric $m\times m$ matrices denoted $\Sym^m$ equipped with inner product
 \begin{equation*}
  \langle X,Y\rangle := \tr\left(X^\trans Y\right),
 \end{equation*}
where $\tr(\cdot)$ (resp. $(\cdot)^\trans$) denote the trace (resp. transpose) of matrix. The induced norm is the \emph{Frobenius norm} which is given by
 \begin{equation*}
  \|X\| = \sqrt{\sum_{i=1}^m\sum_{j=1}^mX_{ij}^2}.
 \end{equation*}
One may, of course, think of the Frobenius norm as treating matrices as ``long vectors". 

The set of \emph{positive (resp. negative) semi-definite} $m\times m$ matrices is denoted $\Sym^m_+$ (resp. $\Sym^m_-$) and we write $x\succeq 0$ (resp. $x\preceq 0$) to mean $x\in\Sym^m_+$ (resp. $x\in \Sym^m_-$).  The set of $m\times m$ \emph{orthogonal} (resp. \emph{permutation}) matrices is denoted $\Orth^m$ (resp. $\Perm^m$).

The \emph{projection} mapping onto the set $\Omega\subseteq\E$ is the set-valued mapping $P_{\Omega}:\E\setto\Omega$ given by
 \begin{equation*}
   P_{\Omega}(x):=\left\{y\in\Omega:\|x-y\|\leq\inf_{z\in\Omega}\|x-z\|\right\}.
 \end{equation*}
When $P_\Omega(x)=\{y\}$ ({\em i.e.,} $P_\Omega(x)$ is a singleton) we write $P_\Omega(x)=y$. 

In finite dimensions the \emph{Mordukhovich normal cone} to the set $\Omega\subseteq\E$ at a point $\overline x\in\Omega$ can be represented as
 $$N_\Omega(\overline x)=\left\{y\in\E:\exists(x_n),(y_n)\text{ s.t. }x_n\to\overline{x},\,y_n\to y,\,y_n\in\mathbb{R}_+(x_n-P_\Omega(x_n))\right\}.$$
For closed convex sets this simplifies to the classical \emph{convex normal cone} given by
 $$N^{\rm conv}_\Omega(\overline{x}):=\{y\in\E:\langle y,x-\overline{x}\rangle\leq 0,\,\forall x\in \Omega\},$$
but still remains useful in non-convex settings \cite[Ch.~1]{mordukhovich2006generalized}. 
The \emph{proximal normal cone} to the set $\Omega\subseteq\E$ at a point $\overline{x}\in\Omega$ is given by
 $$N^{\text{prox}}_\Omega(\overline{x}):=\mathbb{R}_+\left(P^{-1}_\Omega(\overline{x})-\overline{x}\right).$$
A brief summary of relations between the three normal cones is given in the following fact (see, for example, {\cite[Lem.~2.4]{restrictedNormalCones}).
\begin{fact}[Normal cone inclusions]\label{fact:normalInclusion}
 Let $\Omega\subseteq\E$ be non-empty with $\overline{x}\in\Omega$. Then:
  \begin{enumerate}[(a)]
   \item $N^{\rm conv}_\Omega(\overline{x})\subseteq N^{\rm prox}_\Omega(\overline{x}).$
   \item If $\Omega$ is closed, then $N^{\rm prox}_\Omega(\overline{x})\subseteq N_\Omega(\overline{x})$.
   \item If $\Omega$ is closed and convex, then $N^{\rm conv}_\Omega(\overline{x})=N^{\rm prox}_\Omega(\overline{x})=N_\Omega(\overline{x})$.
  \end{enumerate}
\end{fact}

\begin{remark}
 For the full definition of the Mordukhovich normal cone, valid in any Banach space, see \cite[Def.~1.1]{mordukhovich2006generalized}. The above definition is an equivalent characterization which holds in the finite dimensional case \cite[Th.~1.6]{mordukhovich2006generalized}.
\qede\end{remark}

Given $X\in\Sym^m$ denote by $\lambda_j(X)$ the $j$th largest eigenvalue of $X$. In this way,
 $$\lambda_1(X)\geq \lambda_2(X)\geq \dots\geq \lambda_m(X).$$
The \emph{eigenvalue map} is the function $\lambda:\Sym^m\to\mathbb{R}^m$ which maps a symmetric matrix to the $m$-dimensional vector of its eigenvalues arranged in non-increasing order. That is,
 $$\lambda(X):=(\lambda_1(X),\lambda_2(X),\dots,\lambda_m(X)).$$

The \emph{indicator function} of a set $\Omega\subseteq\E$ is the function $\iota_\Omega:\E\to\mathbb{R}\cup\{+\infty\}$ which takes the value $0$ on $\Omega$, and $+\infty$ otherwise. A function $f:\mathbb{R}^m\to\mathbb{R}\cup\{+\infty\}$ is called \emph{symmetric} if $f(x)=f(\sigma x)$ for all $x\in\dom f:=\{x\in\mathbb{R}^m:f(x)<+\infty\}$ and $\sigma\in\Perm^m$. A function $F:\Sym^m\to\mathbb{R}\cup\{+\infty\}$ is called \emph{spectral} if $F(U^\trans XU)=F(X)$ for all $X\in\dom F$ and $U\in\Orth^m$. A subset of $\mathbb{R}^m$ (resp. $\Sym^m$) is said to be \emph{symmetric} (resp. \emph{spectral}) if and only if its indicator function is symmetric (resp. spectral function).

Symmetric and spectral functions have a natural one-to-one correspondence. The relationship is given by
  \begin{equation}\label{eq:fF equiv}
   F(X)=(f\circ \lambda)(X),\qquad f(x)=F(\diag x),
  \end{equation}
where $\diag(x)$ denotes the $m\times m$ diagonal matrix whose diagonal entries are given by the entries of the vector $x$. Consequently, many important properties can be transferred between symmetric and spectral functions \cite{lewis2008proxRegSpectral}. For instance, the following fact shows projections onto spectral sets are easily computed whenever the projection onto the corresponding symmetric set is accessible.

\begin{fact}[Projections onto spectral sets]\label{fact:spectralProjections}
 Let $K\subseteq\mathbb{R}^m$ be a symmetric set. For any $X\in\Sym^m$, the projection of $X$ onto the spectral set $\lambda^{-1}(K)$ is given by
  $$P_{\lambda^{-1}(K)}(X)=\left\{U^\trans(\diag(y))U:y\in P_K\lambda(X),\,U\in\Orth^m(X)\right\},$$
 where the set $\Orth^m(X):=\left\{U\in\Orth^m:X=U^\trans(\diag\lambda(X))U\right\}$.
\end{fact}
\begin{proof}
 We prove only the ``$\subseteq$" inclusion. A proof of the other inclusion can be found in \cite[Th.~21]{lewis2008manifolds}. Suppose $Y\in P_{\lambda^{-1}(K)}(X)$ but $X$ and $Y$ do not have a simultaneously spectral decomposition. Let $X=U^\trans\diag\lambda(X)U$ be an ordered spectral decomposition of $X$. Since $Y\in\lambda^{-1}(K)$, the vector $\lambda(Y)\in K$ and hence  
 $$U^\trans(\diag\lambda(Y))U\in \lambda^{-1}(K).$$  
 By \emph{Fan's inequality} \cite[Th.~1.2.1]{ConvAnalNonOpt} and the orthogonality of $U$,
  \begin{align*}
   \|X-Y\| > \|\lambda(X)-\lambda(Y)\|
   &= \|U^\trans(\diag\lambda(X)-\diag\lambda(Y))U\| \\
   &= \|X-U^\trans(\diag\lambda(Y))U\|.
  \end{align*}
 This implies $Y\not\in P_{\lambda^{-1}(K)}(X)$ which is a contradiction, and completes the proof.
\end{proof}

The various \emph{sub-differentials} of symmetric and spectral functions are also closely related. The following fact states the equivalence in the context of normal cones, which is the important setting for purposes of this paper.

\begin{fact}[Normals to spectral sets]\label{fact:normalSpectralSet}
 Let $K\subseteq\mathbb{R}^m$ be a closed, symmetric set. For any $X\in\lambda^{-1}(K)$, the Mordukhovich normal cone to the spectral set $\lambda^{-1}(K)$ is given by 
 $$N_{\lambda^{-1}(K)}(X)=\{U^\trans(\diag(y))U:y\in N_K(\lambda(X)),\,U\in\Orth^m(X)\}.$$
 The corresponding result for the proximal normal cone also holds.
\end{fact}
\begin{proof}
	Since $K$ is closed and symmetric, its indicator function $\iota_K$ is lower semi-continuous and symmetric. The result follows from \cite[Th.~4.2]{drusvyatskiy215variational}.
\end{proof}

For further details on the interplay between symmetric and spectral functions, the reader is referred to works of Lewis and others \cite{ConvAnalNonOpt,lewis2008proxRegSpectral,eberhard2015,lewis1999nonsmoothEigen,lewis2003eigenOpt,sendovThesis}.

\section{Normal Cones}\label{sec:normalCones}
 In this section we provide a formulae for normal cones to sparsity sets with non-negative constraints. Our approach is to first address the non-negative sparse vector settings, before lifting the results to symmetric matrices.
 
 The symmetric function of central interest in this paper is the \emph{$\ell_0$-functional} denoted ${\|\cdot\|_0:\mathbb{R}^m\to\{0,1,\dots,m\}}$, which counts the number of non-zero entries of a vector. The corresponding spectral function is the matrix \emph{rank} function. Equivalence \eqref{eq:fF equiv} becomes
 $$\rank=\|\cdot\|_0\circ\lambda,\qquad \|\cdot\|_0=(\rank)\circ(\diag).$$
 Let $s\in\{0,1,2,\dots,m\}$. The set of \emph{non-negative sparse vectors} is denoted
 $$\mathcal{K}_s:=\{x\in\mathbb{R}^m_+:\|x\|_0\leq s\}.$$
These can be viewed as the \emph{lower-level sets} of the function from $\mathbb{R}^m\to\mathbb{R}\cup\{+\infty\}$ defined by
  $x\mapsto \|x\|_0+\iota_{\mathbb{R}^m_+}(x)$.
The set of \emph{low-rank positive semi-definite matrices} is denoted
 \begin{equation*}
   \mathcal{S}_{s}:=\{X\in\Sym^m_+:\rank(X)\leq s\}.
 \end{equation*}
 Similarly, these can be viewed as the lower-level sets of the function from $\Sym^m\to\mathbb{R}\cup\{+\infty\}$ defined by
  $X\mapsto \rank(X)+\iota_{\Sym^m_+}(X).$
  
\begin{remark}[$\mathcal{K}_s$ and $\mathcal{S}_s$ are closed sets]\label{re:Ks Ss are closed}
	For all $s\in\{0,1,2,\dots,m\}$, the set $\mathcal{K}_s$  is closed as the intersection of the two closed sets $\mathbb{R}^m_+$ and $\{x\in\mathbb{R}^m:\|x\|_0\leq s\}$. The latter is closed being a lower-level set of the lower semi-continuous function $\|\cdot\|_0$. Similarly, $\mathcal{S}_s$ is closed as the intersection of the closed sets $\Sym^m_+$ and $\{X\in\Sym^m:\rank X\leq s\}$. The latter being a lower-level set lower semi-continuous function $\rank(\cdot)$.
\qede\end{remark}

\subsection{Non-Negative Sparse Vectors}  
Given a vector $x\in\mathbb{R}^m$ we denote $\mathbb{I}(x):=\{j\in\{1,2,\dots,m\}:x_j\neq 0\}$. For convenience, we write $x^+:=P_{\mathbb{R}_+^m}(x)=\max\{0,x\}$ (in the pointwise sense) and $x^-:=P_{\mathbb{R}_-^m}(x)=\min\{0,x\}$ (see \cite{vn93} for further details). The standard basis for $\mathbb{R}^m$ is denoted $e_1,e_2,\dots,e_m$. The set of \emph{sparse vectors} is denoted
 $$\mathcal{A}_s:=\{x\in\mathbb{R}^m:\|x\|_0\leq s\}.$$
The following proposition states, in particular, that the projection onto the set $\mathcal{K}_s$ of a vector is given by a simple thresholding of the vector keeping only its $s$ largest non-negative entries. 

\begin{proposition}[Projection onto $\mathcal{K}_s$ and its inverse]\label{prop:projKs} The following hold.
 \begin{enumerate}[(a)]
  \item $\forall x\in\mathbb{R}^m$ and $\forall y\in P_{\mathcal{K}_s}(x)$, $\Ind(y)\subseteq\Ind(x^+)$.
  \item \label{projii} $\forall x\in\mathbb{R}^m$, $P_{\mathcal{K}_s}(x)=P_{\mathcal{K}_s}(x^+)$.
  \item \label{projiii} $\forall x\in\mathbb{R}^m_+$, $P_{\mathcal{K}_s}(x)=P_{\mathcal{A}_s}(x)$.
  \item \label{projiv} $\forall x\in\mathbb{R}^m$, $P_{\mathcal{K}_s}(x)=P_{\mathcal{A}_s}(x^+)$ and hence
  $$P_{\mathcal{K}_s}(x)=\left\{y\in\mathbb{R}^m:
        y_j = \begin{cases}
                x_j^+, & j\in\mathbb{J},\\
                0,     & j\not\in\mathbb{J};\\
              \end{cases}\text{ for some }\mathbb{J}\in\mathcal{J}_s(x)\right\},$$
    where $$\mathcal{J}_s(x):=\left\{\mathbb{J}\subseteq\{1,2,\dots,m\}:|\mathbb{J}|=s,\,\min_{j\in\mathbb{J}}x_j^+\geq\max_{j\not\in\mathbb{J}}x_j^+\right\}.$$
  \item If $y\in\mathcal{K}_s$ and $\|y\|_0=s$, then
   $$P^{-1}_{\mathcal{K}_s}(y)=\left\{x:y_j=x_j\text{~for all~}j\in\Ind(y),\,\min_{j\in\Ind(y)}y_j\geq \max_{j\not\in\Ind(y)}x_j^+\right\}.$$
  \item If $y\in\mathcal{K}_s$ and $\|y\|_0<s$ then $P^{-1}_{\mathcal{K}_s}(y)=\{x:x^+=y\}=P_{\mathbb{R}^m_+}^{-1}(y)$.
 \end{enumerate}
\end{proposition}
\begin{proof}
 (a) Let $y\in P_{\mathcal{K}_s}(x)$ and suppose there exists an index $j_0\in \Ind(y)\setminus\Ind(x^+)$. Then $y_{j_0}>0$ and $x_{j_0}\leq 0$. Letting $z:=y-y_{j_0}e_{j_0}\in\mathcal{K}_s$ we deduce
   $$\|x-y\|^2=\sum_{j\neq j_0}|x_j-y_j|^2+|x_{j_0}-y_{j_0}|^2>\sum_{j\neq j_0}|x_j-y_j|^2+|x_{j_0}-0|^2=\|x-z\|^2,$$
   which contradicts the assumption $y\in P_{\mathcal{K}_s}(x)$.
   
 (b) Let $y\in P_{\mathcal{K}_s}(x)\cup P_{\mathcal{K}_s}(x^+)$ be arbitrary. By (a),     
     $\Ind(y)\subseteq\Ind(x^+)$, hence $\langle x^+-y,x^-\rangle=0$ and
      $$\|x-y\|^2=\|x^+-y\|^2+\|x^-\|^2.$$
     This implies
      $$\argmin_{y\in\mathcal{K}_s}\|x-y\|=\argmin_{y\in\mathcal{K}_s}\|x^+-y\|,$$
     from which the result follows.

  (c) Since $\mathcal{A}_s\cap\mathcal{K}_s=\mathcal{K}_s$, it suffices to show that any $y\in P_{\mathcal{A}_s}(x)$ is contained in $\mathcal{K}_s$. To this end, suppose $y\in P_{\mathcal{A}_s}(x)\setminus\mathcal{K}_s$ and let $j_0$ be an index such that $y_{j_0}<0$. Letting $z:=y-y_{j_0}e_{j_0}\in-\mathcal{K}_s$, since $x\in\mathbb{R}_+^m$ we have
  $$\|x-y\|^2 > \sum_{j\neq j_0}|x_j-y_j|^2+|x_j-0|^2=\|x-z\|^2,$$
  which contradicts the assumption that $y\in P_{\mathcal{A}_s}(x)$.
     
  (d) Follows from \eqref{projii}, \eqref{projiii} and \cite[Prop.~3.6(ii)]{restrictedNormalConesAffine}.
  (e) Follows from \eqref{projiv} and \cite[Prop.~3.6(v)]{restrictedNormalConesAffine}.
  (f) Follows from \eqref{projiv} and \cite[Prop.~3.6(vi)]{restrictedNormalConesAffine}.
\end{proof}

Given a vector $x\in\mathbb{R}^m$ denote by $[x]$ the vector in $\mathbb{R}^m$ obtained by permuting the entries of $x$ in non-increasing order. Under this notation, we note that $[x]_j$, the $j$th coordinate of the vector $[x]$, is the $j$th largest entry in the vector $x$.  For vectors $x,y\in\mathbb{R}^m$, $x\odot y$ denotes the \emph{Hadamard product} given pointwise by $(x\odot y)_j:=x_jy_j$ for all $j\in\{1,2,\dots,m\}$. It is worth noting that for fixed $\overline{x}$, the set
 $\{y\in\mathbb{R}^m:\overline{x}\odot y=0\},$
is simply the perpendicular subspace to the support of $\overline{x}$.

The following lemma is a kind of (non-convex) analogue to \emph{Moreau's decomposition theorem} \cite[Th.~6.29]{monotoneOperator}, which applies to convex cones, for the set $\mathcal{K}_s$. 

\begin{lemma}[Decomposition lemma for $\mathcal{K}_s$]\label{lem:decompKs}
 Let $x,y,z\in\mathbb{R}^m$ with $x=y+z$. Then $y\in P_{\mathcal{K}_s}x$ if and only if $y\in\mathcal{K}_s, y\odot z=0$ and $[y]_s\geq[z]_1$.
\end{lemma}
\begin{proof}
 Suppose $y\in P_{\mathcal{K}_s}x\subseteq\mathcal{K}_s$. By Proposition~\ref{prop:projKs} there is an index set $\mathbb{J}_0\in\mathcal{J}_s(x)$ such that
  $$y_j=\begin{cases}
         x_j^+, & j\in\mathbb{J}_0,\\
         0,     & j\not\in\mathbb{J}_0;\\
        \end{cases}\qquad\text{and}\qquad 
        \min_{j\in\mathbb{J}_0}x_j^+ \geq \max_{j\not\in\mathbb{J}_0}x_j^+.$$
  Thus $z$ is given pointwise by
   $$z_j=x_j-y_j=\begin{cases}
                           x_j^-, & j\in\mathbb{J}_0,\\
                           x_j,     & j\not\in\mathbb{J}_0.\\
                         \end{cases}.$$
  It follows that $[y]_s\geq [z]_1$ since
   \begin{align*}
     [y]_s = \min_{j\in\mathbb{J}_0}x_j^+ &\geq \max_{j\not\in\mathbb{J}_0}x_j^+ \\
       &\geq \max \left(\{x_j^+:j\not\in\mathbb{J}_0\}\cup\{x_j^-:j\in\mathbb{J}_0\}\right) \\
       &\geq \max \left(\{x_j\;:j\not\in\mathbb{J}_0\}\cup\{x_j^-:j\in\mathbb{J}_0\}\right)=[z]_1.
   \end{align*}
  To show $y\odot z=0$, observe that $(y\odot z)_j=y_jz_j$ is either $x_j^+x_j^-$ if $j\in\mathbb{J}_0$, or $0\cdot x_j$ if $j\not\in\mathbb{J}_0$, which is zero in either case.
  
  Conversely, suppose $y\in\mathcal{K}_s, y\odot z=0$ and $[y]_s\geq [z]_1$. Since $x=y+z$ and $y\odot z=0$, for each $j\in\{1,2,\dots,m\}$ either $y_j=x_j$ and $z_j=0$; or $y_j=0$ and $z_j=x_j$. Since $y\in\mathcal{K}_s$ there is an index set $\mathbb{J}_0$ with $|\mathbb{J}_0|\leq s$ such that we may express
   \begin{equation}\label{eq:projConverse}
   y_j=\begin{cases}
          x_j^+, & j\in\mathbb{J}_0, \\
          0 ,    & j\not\in\mathbb{J}_0; \\
         \end{cases},\qquad
   z_j=\begin{cases}
          x_j^-, & j\in\mathbb{J}_0, \\
          x_j,   & j\not\in\mathbb{J}_0; \\
         \end{cases}
   \end{equation}
   noting that $x_j=x_j^+$ and $x_j^-=0$ for $j\in\mathbb{J}_0$.
   
   To show that $y\in P_{\mathcal{K}_s}(x)$ it suffices to consider the case in which $|\mathbb{J}_0|=s$. For if $|\mathbb{J}_0|<s$ then $[y]_s=0$, hence
   $$0=[y]_s\geq [z]_1 = \max\left(\{x_j^-:j\in\mathbb{J}_0\}\cup\{x_j:j\not\in\mathbb{J}_0\}\right).$$
  In particular, $x_j\leq 0$ for $j\not\in\mathbb{J}_0$, or equivalently $x_j=x_j^-$  for $j\not\in\mathbb{J}_0$. It is therefore possible to replace the index set $\mathbb{J}_0$ with a superset having cardinality $s$ without changing \eqref{eq:projConverse}.
  
   Thus, suppose $|\mathbb{J}_0|=s$ but $\mathbb{J}_0\not\in\mathcal{J}_s(\overline{x})$. Then there exist indices $j_1\in\mathbb{J}_0$ and $j_2\not\in\mathbb{J}_0$ such that $x_{j_1}^+<x_{j_2}^+$. In particular, $x_{j_2}=x_{j_2}^+>0$ and hence
    $$[y]_s=\min_{j\in\mathbb{J}_0}x_j^+\leq x_{j_1}^+<x_{j_2}^+=x_{j_2} \leq \max\left(\{x_j^-:j\in\mathbb{J}_0\}\cup\{x_j:j\not\in\mathbb{J}_0\}\right)=[z]_1.$$
 This contradicts the assumption that $[y]_s\geq [z]_1$ and we therefore conclude that $\mathbb{J}_0\in\mathcal{J}_s(x)$. Proposition~\ref{prop:projKs} now implies $y\in P_{\mathcal{K}_s}(x)$, and thus completes the proof.
\end{proof}

We now provide our first main result: a novel characterization of the Mordukhovich normal cone to the set of non-negative sparse vectors. Given $y\in\mathbb{R}^m$ and an index set $\mathbb{J}\subseteq\{1,\dots,m\}$ the notation $y|_\mathbb{J}= 0$ means $y_j= 0$ for all $j\in\mathbb{J}$.

\begin{theorem}[Mordukhovich normal cone to $\mathcal{K}_s$]\label{th:NKs}
 The Mordukhovich normal cone to the set $\mathcal{K}_s$ at a point $\overline x\in\mathcal{K}_s$ is given by
  $$N_{\mathcal{K}_s}(\overline{x})=\left\{y\in\mathbb{R}^m:\overline{x}\odot y=0,\,y\leq 0\right\}\cup\left\{y\in\mathbb{R}^m:\overline{x}\odot y=0,\,\|y\|_0\leq m-s\right\}.$$
\end{theorem}
\begin{proof}
 ($\subseteq$) Suppose $y\in N_{\mathcal{K}_s}(\overline{x})$. Then there exists sequences $(x_k),(y_k),(z_k)\subseteq\mathbb{R}^m$ such that 
  $$x_k\to\overline{x},\qquad y_k=\alpha_kz_k\to y,\qquad z_k=x_k-p_k,$$
 where $\alpha_k\in\mathbb{R}_+$ and $p_k\in P_{\mathcal{K}_s}(x_k)$. By Lemma~\ref{lem:decompKs}, $p_k\odot z_k=0$ and thus
  $$x_k\odot y_k=(z_k+p_k)\odot (\alpha_kz_k) = \alpha_k (z_k\odot z_k) = z_k\odot y_k.$$
The definition  of $P_{\mathcal{K}_s}$ implies that $\|z_k\|=d(x_k,\mathcal{K}_s)$. Hence, by noting that $x_k\to\overline{x}\in\mathcal{K}_s$ and that the function $d(\cdot,\mathcal{K}_s)$ is continuous, we deduce that $z_k\to 0$. Altogether
  \begin{align*}
   \overline{x}\odot y =\left(\lim_{k\to\infty}x_k\right)\odot \left(\lim_{k\to\infty}y_k\right) 
   &=\lim_{k\to\infty}\left(x_k\odot y_k\right) \\
   &=\lim_{k\to\infty}\left(z_k\odot y_k\right)\\
   &=\left(\lim_{k\to\infty}z_k\right)\odot\left(\lim_{k\to\infty} y_k\right)
   =0\odot y=0.
  \end{align*}
 	By Proposition~\ref{prop:projKs}\eqref{projiv}, for each $k\in\mathbb{N}$, there is an index set $\mathbb{J}_k\in\mathcal{J}_s(x_k)$ such that $p_k$ is of the form
 	  $$ (p_k)_j= \begin{cases}
	            	  (x_k)_j^+, & j\in\mathbb{J}_k,\\
				 	  0,     & j\not\in\mathbb{J}_k.\\
			 	  \end{cases}. $$
	The collection $\{\mathbb{J}_k:k\in\mathbb{N}\}$ is finite, and thus, by the Pigeonhole Princple, there exists a subsequence $({k_l})$ and an index set $\mathbb{J}_0\in\mathcal{J}_s(x_{k_l})$ such that
  \begin{equation}\label{eq:y_kl}
   (p_{k_l})_j = \begin{cases}
                   (x_{k_l})_j^+, & j\in\mathbb{J}_0,     \\
                   0,             & j\not\in\mathbb{J}_0; \\
                  \end{cases} \qquad
    (y_{k_l})_j = \begin{cases}
                   (x_{k_l})_j^-, & j\in\mathbb{J}_0,     \\
                   (x_{k_l})_j,   & j\not\in\mathbb{J}_0. \\
                  \end{cases}
  \end{equation}
 Since $\mathbb{J}_0\in\mathcal{J}_s(x_{k_l})$, for all $l\in\mathbb{N}$ we have
   $$\min_{j\in\mathbb{J}_0}(x_{k_l})_j^+ \geq \max_{j\not\in\mathbb{J}_0}(x_{k_l})_j^+,$$
  and hence
  $$\min_{j\in\mathbb{J}_0}\overline{x}^+_j = \lim_{l\to\infty}\left(\min_{j\in\mathbb{J}_0}(x_{k_l})_j^+\right) \geq \lim_{l\to\infty}\left(\max_{j\not\in\mathbb{J}_0}(x_{k_l})_j^+\right)=\max_{j\not\in\mathbb{J}_0}\overline{x}^+_j.$$
We therefore conclude that $\mathbb{J}_0\in\mathcal{J}_s(\overline{x})$. If $y\not\leq 0$ then there is an index $j_0\in\{1,2,\dots,m\}$ such that $y_{j_0}>0$. Since $y_k\to y$, we assume $l$ to be sufficiently large so that $(y_{k_l})_{j_0}>0$. From the representation of $y_{k_l}$ in \eqref{eq:y_kl} we deduce that $j_0\not\in\mathbb{J}_0$ and $(x_{k_l})_{j_0}=(y_{k_l})_{j_0}>0$. By the definition of $\mathbb{J}_0$, it follows that $$(x_{k_l})_j\geq(x_{k_l})_{j_0}>0 \implies (x_{k_l})_j^-=0,$$
for all $j\in\mathbb{J}_0$. By \eqref{eq:y_kl} we deduce $y_{k_l}|_{\mathbb{J}_0}=0$. 
Since $y_{k_l}\to y$ it follows that $y|_{\mathbb{J}_0}=0$, and therefore $\|y\|_0\leq m-s$.  
  
 ($\supseteq$) Suppose $\overline{x}\odot y=0$, or equivalently, for each $j\in\{1,2,\dots,m\}$, we have that
 \begin{equation}\label{eq:compl}
 	\text{$\overline{x}_j$ and $y_j$ cannot be simultaneously non-zero.}
 \end{equation}
 
 For the $y\leq 0$ case, define sequences $(x_k)$ and $(y_k)$ by
  $$x_k:=\overline{x}+\frac{1}{k}yto\overline{x},\qquad y_k\in k(x_k-P_{\mathcal{K}_s}(x_k)).$$
 Then, by noting \eqref{eq:compl}, for any $k\in\mathbb{N}$, we have $P_{\mathcal{K}_s}(x_k)=\{\overline{x}\}$ thus  $$y_k=k\left(\left(\overline{x}+\frac{y}{k}\right)-\overline{x}\right)=y\in N_{\mathcal{K}_s}(\overline{x}).$$
  
  For the other case, suppose $\|y\|_0\leq m-s$. Then, using \eqref{eq:compl}, we see that there exists an index set $\mathbb{J}_0\in\mathcal{J}_s(\overline{x})$ such that $y|_{\mathbb{J}_0}=0$. Let $w\in\mathbb{R}^m$ be the vector whose entries are $1$ on $\mathbb{J}_0$, and $0$ otherwise. Define sequences $(x_k)$ and $(y_k)$ by
  $$x_k:=\overline{x}+\frac{1}{k}y+\frac{1}{\sqrt{k}}w\to\overline{x},\qquad y_k\in k(x_k-P_{\mathcal{K}_s}(x_k)).$$
 Since $1/k\to 0$ at a faster rate than $1/\sqrt{k}\to0$, there exists a sufficiently large $K$ such that for $k>K$,
  $$\min_{j\in\mathbb{J}_0}\left\{\overline{x}_j+\frac{1}{\sqrt{k}}w_j\right\} > \frac{1}{k}\max_{j\not\in\mathbb{J}_0}\left\{y_j\right\}.$$
  Hence for $k>K$ we have  $P_{\mathcal{K}_s}(x_k)=\{\overline{x}+w/\sqrt{k}\}$, and thus that
   $$y_k=k\left(\left(\overline{x}+\frac{1}{k}y+\frac{1}{\sqrt{k}}w\right)-\left(\overline{x}+\frac{1}{\sqrt{k}}w\right)\right)=y\in N_{\mathcal{K}_s}(\overline{x}).$$
  This completes the proof.
\end{proof}

For convex sets the convex and Mordukhovich normal cones coincide \cite[Ch.~1]{mordukhovich2006generalized}. As an immediate consequence of Theorem~\ref{th:NKs} we recover the following well-known result.

\begin{corollary}[Normal cone to $\mathbb{R}^m_+$]\label{cor:NR+s}
 The normal cone to the convex set $\mathbb{R}^m_+$ at the point $\overline{x}\in\mathbb{R}^m_+$ is given by
  $$N_{\mathbb{R}^m_+}(\overline{x})=N_{\mathbb{R}^m_+}^{\rm conv}(\overline{x})=\{y\in\mathbb{R}^m:\overline{x}\odot y=0,\,y\leq 0\}.$$
\end{corollary}
\begin{proof}
 Note that $\mathbb{R}_+^m=\mathcal{K}_m$ and $\left\{y\in\mathbb{R}^m:\overline{x}\odot y=0,\,\|y\|_0\leq 0\right\}=\{0\}$. The result now follows by applying Theorem~\ref{th:NKs} with $s=m$.
\end{proof}

The following remark sheds light on the two sets in the normal cone formula of Theorem~\ref{th:NKs}.

\begin{remark}[$N_{\mathcal{K}_s}$ is the union of two normal cones]\label{re:NKs}
 The Mordukhovich normal to sparsity set $\mathcal{A}_s$ is given by (see \cite[Th.~3.9]{restrictedNormalConesAffine})
  \begin{equation}\label{eq:NAs}
   N_{\mathcal{A}_s}(\overline{x})=\left\{y\in\mathbb{R}^m:\overline{x}\odot y=0,\,\|y\|_0\leq m-s\right\}.
  \end{equation}
 Combining with Corollary~\ref{cor:NR+s}, Theorem~\ref{th:NKs} can be expressed
  $$N_{\mathcal{K}_s}(\overline{x})=N^{\rm conv}_{\mathbb{R}^m_+}(\overline{x})\cup N_{\mathcal{A}_s}(\overline{x}).$$
 That is, $N_{\mathcal{K}_s}$ is the union of the convex normal cone to the the non-negativity set $\mathbb{R}^m_+$, and the Mordukhovich normal cone to the sparsity set $\mathcal{A}_s$. 
\qede\end{remark}

\begin{remark}\label{re:failure of intersection rule}
Note that formulae for $N_{\mathbb{R}^m_+}$ and  $N_{\mathcal{A}_s}$ are known, and that $\mathcal{K}_s=\mathcal{A}_s\cap\mathbb{R}^m_+$. Nevertheless, it is not possible to obtain the the normal cone $N_{K_s}$ using the standard \emph{intersection rule} \cite[\S3.1.1]{mordukhovich2006generalized} applied to $N_{\mathcal{A}_s}$ and $N_{\mathbb{R}^m_+}$ since the basic qualification condition $N_{\mathcal{A}_s}(\overline{x})\cap(-N_{\mathbb{R}^m_+})(\overline{x})=\{0\}$ is not satisfied.
\qede\end{remark}

Around points of maximal sparsity the set $\mathcal{K}_s$ is locally indistinguishable from the sparsity set $\mathcal{A}_s$. In this case, the formula for the normal cone simplifies accordingly.
\begin{corollary}[Points of maximal sparsity]\label{cor:NKsMaximal}
 The Mordukhovich normal cone to the set $\mathcal{K}_s$ at a point $\overline x\in\mathcal{K}_s$ having $\|x\|_0=s$ is given by
  $$N_{\mathcal{K}_s}(\overline{x})=N_{\mathcal{A}_s}(\overline{x})=\left\{y\in\mathbb{R}^m:\overline{x}\odot y=0\right\}.$$
\end{corollary}
\begin{proof}
 Since $\|x\|_0=s$ we have
  $\left\{y\in\mathbb{R}^m:\overline{x}\odot y=0,\,\|y\|_0\leq m-s\right\}=\{y:\overline{x}\odot y=0\}.$
 Observe,
  $$\left\{y\in\mathbb{R}^m:\overline{x}\odot y=0,\,y\leq 0\right\}\subseteq \left\{y\in\mathbb{R}^m:\overline{x}\odot y=0\right\}.$$
 The result now follows from Theorem~\ref{th:NKs} and \eqref{eq:NAs}.
\end{proof}

To conclude our study of the vector setting, we give a characterization of the proximal normal cone to $\mathcal{K}_s$, in terms of already introduced objects.

\begin{theorem}[Proximal normal cone to $\mathcal{K}_s$]\label{th:NproxKs}
 The proximal normal cone to $\mathcal{K}_s$ at the point $\overline{x}\in\mathcal{K}_s$ is given by
  $$N^{{\rm prox}}_{\mathcal{K}_s}(\overline{x})
      = \begin{cases}
         N_{\mathbb{R}^m_+}^{\rm conv}(\overline{x})=N_{\mathbb{R}^m_+}(\overline{x}), & \|\overline{x}\|_0<s, \\
         N_{\mathcal{A}_s}^{\rm prox}(\overline{x})=N_{\mathcal{A}_s}(\overline{x})=N_{\mathcal{K}_s}(\overline{x}),  & \|\overline{x}\|_0=s. \\
        \end{cases}$$
\end{theorem}
\begin{proof}
 On one hand, if $\|\overline{x}\|_0<s$ then, by Proposition~\ref{prop:projKs}, $P^{-1}_{\mathcal{K}_s}(\overline{x})=P^{-1}_{\mathbb{R}^m_+}(\overline{x})$ implying that the corresponding proximal normal cones coincide. The claimed formula now follows by Fact~\ref{fact:normalInclusion}. On the other hand, if $\|\overline{x}\|_0=s$ the result follows from Corollary~\ref{cor:NKsMaximal} and \cite[Prop.~3.8]{restrictedNormalConesAffine}.
\end{proof}

\begin{remark}
 Theorem~\ref{th:NproxKs} shows that the proximal normal cone to a non-negative sparsity set does not capture all important features of the set at points which do not have maximal sparsity. In particular, the proximal normal cone coincides with the normal cone to the convex non-negative cone, whereas the corresponding Mordukhovich cone need not. Similar behavior is observed for the proximal normal cone to the sparsity set $\mathcal{A}_s$, which is equal to $\{0\}$ at points not having maximal sparsity \cite[Prop.~3.8]{restrictedNormalConesAffine}.
\qede\end{remark}

\subsection{Low-Rank Positive Semi-Definite Matrices}
Using the correspondence between symmetric and spectral functions, we now \emph{lift} our vector results to the larger space of symmetric matrices.

\begin{proposition}[Projection onto $\mathcal{S}_s$]\label{prop:projSs}
 Let $X\in\Sym^m$ and define
  $$\lambda^+_s(X):=(\lambda_1^+(X),\dots,\lambda_s^+(X),0,\dots,0).$$
The projection of $X$ onto $\mathcal{S}_s$ is given by 
  $$P_{\mathcal{S}_s}(X)=\left\{Y\in\Sym^m:Y=U^\trans(\diag\lambda^+_s(X))U,\,U\in\Orth^m(X)\right\}.$$
\end{proposition}
\begin{proof}
 Follows from Fact~\ref{fact:spectralProjections} and Proposition~\ref{prop:projKs}.
\end{proof}

Our next main result is a novel characterization of the Mordukhovich normal cone to the set $\mathcal{S}_s$.

\begin{theorem}[Mordukhovich normal cone to $\mathcal{S}_s$]\label{th:NSs}
 The Mordukhovich normal cone to the set $\mathcal{S}_s$ at the point $\overline{X}\in\mathcal{S}_s$ is given by
   $$N_{\mathcal{S}_s}(\overline X)=\{Y\in\Sym^m:\overline XY=0,\,Y\preceq 0\}\cup   \{Y\in\Sym^m:\overline XY=0,\,\rank(Y)\leq m-s\}.$$
\end{theorem}
\begin{proof}
	First observe that, by combining Theorem~\ref{th:NKs} and Fact~\ref{fact:normalSpectralSet}, we obtain
	\begin{equation}\label{eq:NSs inter}\begin{split}
		N_{\mathcal{S}_s}(\overline{X}) =&\left\{U^\trans \diag(y)U:\lambda(\overline{X})\odot y=0,\,y\leq 0,\,U\in\Orth^m(X)\right\}\\
		&\quad\cup\left\{U^\trans \diag(y)U:\lambda(\overline{X})\odot y=0,\,\|y\|_0\leq m-s,\,U\in\Orth^m(X)\right\}.
	\end{split}\end{equation} 
	 To complete the proof, we show that \eqref{eq:NSs inter} is equal to the claimed formula. Since other inclusion is easily deduced, we only prove that any matrix $Y$ satisfying the proposed formula is contained in \eqref{eq:NSs inter}. To this end, consider $Y\in\Sym^m$ with $\overline{X}Y=0$. Then $\overline{X}$ and $Y$ commute, and hence have a simultaneously spectral decomposition which we assume, without loss of generality, is an ordered spectral decomposition for $\overline{X}$. That is, there exists $U\in\Orth^m(\overline{X})$ such that 
	 $$Y=U^\trans\diag(y)U\text{ for some }y\in\mathbb{R}^m.$$
	Furthermore, since $U\in\Orth^m$, we have that
	  \begin{align*}
	    \diag(\lambda(\overline{X})\odot y) 
	    &= \diag(\lambda(\overline{X}))\diag(y) \\
	    &= U \left(U^\trans\diag(\lambda(\overline{X}))U\right)\left(U^\trans \diag(y)U\right) U^\trans \\
	    &= U\overline{X}YU^\trans =0,
	  \end{align*}
	which implies that $\lambda(X)\odot y=0$. To complete the proof, we note that $Y\preceq 0$ implies $y\leq 0$ and that $\rank Y\leq m-s$ implies $\|y\|_0\leq m-s$.
\end{proof}

As before, we deduce consequences of Theorem~\ref{th:NSs}. The first is the normal cone to the set of positive semi-definite matrices. This can be found, for example, in \cite{malick2012fresh}.
\begin{corollary}[Normal cone to $\Sym^m_+$]\label{cor:NS+}
 The normal cone to the set $\Sym^m_+$ at a point $\overline{X}\in\Sym^m_+$ is given by
  $$N_{\Sym^m_+}(\overline{X})=N_{\Sym^m_+}^{\rm conv}(\overline{X})=\{Y\in\Sym^+:\overline{X}Y=0,\,Y\preceq 0\}.$$
\end{corollary}
\begin{proof}
 Note that $\Sym^m_+=\mathcal{S}_s$ and $\{Y\in\Sym^m:\overline XY=0,\,\rank(Y)\leq 0\}=\{0\}$. The result now follows from Theorem~\ref{th:NSs}  with $s=m$.
\end{proof}

Denote the set of low-rank symmetric matrices by
 $$\mathcal{R}_s:=\left\{X\in\Sym^m:\rank(X)\leq s\right\}.$$
The following proposition is a characterization of the Mordukhovich normal cone to $\mathcal{R}_s$.
\begin{proposition}[Mordukhovich normal cone to $\mathcal{R}_s$]\label{prop:NRs}
 The Mordukhovich normal cone to the set $\mathcal{R}_s$ at a $\overline{X}\in\mathcal{R}_s$ having $\rank\overline{X}=s$ is given by
   $$N_{\mathcal{R}_s}(\overline{X})=\{Y\in\Sym^m:\overline XY=0\}.$$
\end{proposition}
\begin{proof}
 Follows from \eqref{eq:NAs} and Fact~\ref{fact:normalSpectralSet}.
\end{proof}

\begin{remark}
 A formula for the Mordukhovich normal cone to the set of low-rank (possibly rectangular) real matrices is derived in \cite{luke2013prox}. This formula cannot be applied in our context even, when specialized to the square case, for the same reasons discussed in Remark~\ref{re:failure of intersection rule}
\qede\end{remark}

As was the case in the vector setting, here we are also able to express the normal cone to $\mathcal{S}_s$ in terms of the normal cones of its two `building blocks'.

\begin{remark}[$N_{\mathcal{S}_s}$ is the union of two normal cones]
 In light of Theorem~\ref{th:NSs}, Corollary~\ref{cor:NS+} and Proposition~\ref{prop:NRs}, the Mordukhovich normal cone to $\mathcal{S}_s$ can be expressed as
   $$N_{\mathcal{S}_s}(\overline{X})=N^{\rm conv}_{\Sym^m_+}(\overline{X})\cup N_{\mathcal{R}_s}(\overline{X}).$$
 That is, $N_{\mathcal{S}_s}$ is the union of the convex normal cone to the positive semi-definite matrices, and the Mordukhovich normal cone to the low-rank set $\mathcal{R}_s$.
\qede\end{remark}

A characterization of the proximal normal cone to $\mathcal{S}_s$ can also be given. As was the case in the vector setting, it also does not adequately describe the geometry of $\mathcal{S}_s$.

\begin{theorem}[Proximal normal cone to $\mathcal{S}_s$]
 The proximal normal cone to $\mathcal{S}_s$ at the point $\overline{X}\in\mathcal{K}_s$ is given by
  $$ N^{{\rm prox}}_{\mathcal{S}_s}(\overline{X})
      = \begin{cases}
         N_{\Sym^m_+}^{\rm conv}\left(\overline{X}\right)=N_{\Sym^m_+}\left(\overline{X}\right),       & \rank(\overline{X})<s, \\
         N_{\mathcal{R}_s}^{\rm prox}\left(\overline{X}\right)\hspace{0.15ex}= N_{\mathcal{R}_s}\left(\overline{X}\right),  & \rank(\overline{X})=s. \\
        \end{cases}$$
\end{theorem}
\begin{proof}
 Follows from Fact~\ref{fact:normalSpectralSet} and Theorem~\ref{th:NproxKs}.
\end{proof}

\section{Regularity Properties}\label{sec:regular}
 We now investigate regularity properties of non-negative sparsity sets, and collections of sets containing non-negative sparsity sets. We first recall some definitions.

A family $\{\Omega_1,\Omega_2,\dots,\Omega_n\}$ of closed non-empty subsets of $\E$ is \emph{strongly regular} at a point $\overline{x}\in\bigcap_{k=1}^n\Omega_k$ if $y_1=y_2=\dots=y_m=0$ is the only solution to the system
 $$\sum_{k=1}^my_k=0,\quad y_k\in N_{\Omega_k}(\overline{x}).$$
Specialized to a collection of two sets, strong regularity can be expressed as the transversality condition
 \begin{equation}\label{eq:strong regularity two sets}
    N_{\Omega_1}(\overline{x})\cap\left(-N_{\Omega_2}(\overline{x})\right)=\{0\}.
 \end{equation}
 The normal cone formulae in Theorems~\ref{th:NKs} and \ref{th:NSs} provide conditions for strongly regularity of intersections involving non-negative sparsity sets. These condition, which we derive below, are analogous to those given by \cite[Prop.~3.8]{luke2013prox} for rank constraints not requiring non-negativity.

\begin{proposition}[Strong regularity for intersections with $\mathcal{K}_s$]\label{prop:strongRegKs}
 Let $\Omega\subseteq\mathbb{R}^m$ be closed and $\overline{x}\in\Omega\cap\mathcal{K}_s$. Then $\{\Omega,\mathcal{K}_s\}$ is strongly regular at $\overline{x}$ if and only if for all non-zero $y\in N_{\Omega}(\overline{x})$ either (a) $\overline{x}\odot y\neq 0$, or (b) $y\not\geq 0$ and $\|y\|_0>m-s$.
\end{proposition}
\begin{proof}
 By characterization of $N_{\mathcal{K}_s}$ in Theorem~\ref{th:NKs}, we deduce that strong regularity of the sets $\{\Omega,\mathcal{K}_s\}$ is equivalent to
  \begin{align*}
   \{0\}&=N_{\Omega}(\overline{x})\cap \{y\in\mathbb{R}^m:\overline{x}\odot y=0,\,y\geq 0\}, \text{~and~}\\
   \{0\}&=N_{\Omega}(\overline{x})\cap \{y\in\mathbb{R}^m:\overline{x}\odot y=0,\,\|y\|_0\leq (m-s)\}.
  \end{align*}
 The result now follows.
\end{proof}

The following Proposition is the symmetric matrix analogue of Proposition~\ref{prop:strongRegKs}.

\begin{proposition}[Strong regularity for intersections with $\mathcal{S}_s$]\label{prop:strongRegSs}
 Let $\Omega\subseteq\Sym^m$ be closed and $\overline{X}\in\Omega\cap\mathcal{S}_s$. Then $\{\Omega,\mathcal{S}_s\}$ is strongly regular at $\overline{X}$ if and only if for all non-zero $Y\in N_{\Omega}(\overline{X})$ either (a) $\overline{X}Y\neq 0$, or (b) $Y\not\succeq 0$ and $\rank(Y)>m-s$.
\end{proposition}
\begin{proof}
 Argue similarly to Proposition~\ref{prop:strongRegKs}, using the characterization of $N_{\mathcal{S}_s}$ in Theorem~\ref{th:NSs}.
\end{proof}

In Section~\ref{sec:linear system}, we give an application of Proposition~\ref{prop:strongRegKs} to solving sparse linear systems, and of Proposition~\ref{prop:strongRegSs} to low-rank semi-definite programming feasibility.

\begin{remark}[Affine-hull regularity]\label{re:affine-hull regularity}
 \emph{Affine-hull regularity} is a regularity notion for collections of sets which is weaker than strong regular which has been utilized to obtain convergence results in the absence of strong regularity \cite{restrictedNormalCones,phan2014linearDR}. However, as we now explain, for collections of constraint sets containing a sparsity set, the two notions coincide.
			
 Focusing on the case of two sets, affine-hull regularity can be viewed as a modification of the definition of strong regularity in which the Mordukovich normal cones in \eqref{eq:strong regularity two sets} are replaced with the so-called \emph{restricted Mordukovich normal cones}, developed in \cite{restrictedNormalCones,restrictedNormalConesAffine}, with the restriction performed with respect to the affine-hull of the union of the two constraint sets. If $s\geq 1$ and the affine-hull of the any of sparsity sets $\mathcal{K}_s$, $\mathcal{A}_s$, $\mathcal{S}_s$ or $\mathcal{R}_s$ is equal to the entire space. The corresponding affine restricted Mordukovich normal and standard Mordukhovich normal cones coincide and, it therefore follows that, for a collection of constraint sets containing such a sparsity set, the notions of strong and affine-hull regularity coincide.
\qede\end{remark}

A closed set $\Omega\subseteq\E$ is \emph{prox-regular} at $\overline{x}\in\Omega$ if for all $\overline{y}\in N_\Omega(\overline{x})$ there exists an $\epsilon>0$ and $\rho>0$ such that whenever $\|x-\overline{x}\|<\epsilon$ and $y\in N_\Omega(x)$ with $\|y-\overline{y}\|<\epsilon$ then $x$ is the unique nearest point of $\{z\in\Omega:\|z-\overline{x}\|<\epsilon\}$ to $x+\rho y$. A useful characterization which we utilize in the following proofs is that prox-regularity at $\overline{x}\in\Omega$ is equivalent to single-valuedness of $P_{\Omega}$ around $\overline{x}$ \cite[Th.~1.3]{rockafellar2000local}. 

\begin{proposition}[Prox-regularity of $\mathcal{K}_s$]\label{prop:proxRegKs}
 Let $\overline{x}\in \mathcal{K}_s$ with $s\in\{1,\dots,(m-1)\}$ and $m\geq 2$. Then $\mathcal{K}_s$ is prox-regular at $\overline{x}$ if and only if $\|\overline{x}\|_0=s$.
\end{proposition}
\begin{proof} 
 On one hand, suppose $\|\overline{x}\|_0<s$. To show that $\mathcal{K}_s$ is not prox-regular at $\overline{x}$, it suffices to produce a sequence $x_k\to\overline{x}$ such that $P_{\mathcal{K}_s}(x_k)$ is not singleton for any $k\in\mathbb{N}$. To this end, since $\|\overline{x}\|_0\leq (s-1)$, we have
  $$2\leq(s-\|\overline{x}\|_0)+1=(s+1)-\|\overline{x}\|_0\leq m-\|\overline{x}\|_0.$$
 Since $x$ has $(m-\|x\|_0)$ entries which are zero, therefore there exists an index set $\mathbb{I}_0\subseteq \{1,\dots,m\}$ with $|\mathbb{I}_0|=(s-\|\overline{x}\|_0)+1\geq 2$ such that $\overline{x}_j=0$ for all $j\in\mathbb{I}_0$. Define the vector $v\in\mathbb{R}^m$ and the sequence $(x_k)$ by
  $$v:=\sum_{j\in\mathbb{I}_0}e_{j},\qquad x_k:=\overline{x}+\frac{1}{k}v.$$
 Observe $\|x_k\|_0=\|\overline{x}\|_0+\|v\|_0=s+1$. Thus, for sufficiently large $k$, Proposition~\ref{prop:projKs} yields
$$P_{\mathcal{K}_s}(x_k)=\left\{\overline{x}+\frac{1}{k}(v-e_j):j\in\mathbb{I}_0\right\}\text{ where }|P_{\mathcal{K}_s}(x_k)|=|\mathbb{I}_0|\geq 2.$$
 By taking $k$ sufficiently large, the distance between $x_k$ and $\overline{x}$ can be made arbitrary small, thus proving the that $\mathcal{K}_s$ is not prox-regular at $\overline{x}$. 
 
 On the other hand, suppose $\|\overline{x}\|_0=s$, and let $\delta:=\frac{1}{2}\min\{\overline{x}_j:\overline{x}_j>0\}>0$. Since $\overline{x}$ has maximal sparsity, the set $\mathcal{J}_s(\overline{x})$ is a singleton, say $\mathcal{J}_s(\overline{x})=\{\mathbb{J}_0\}$. For any $x\in \mathbb{B}_\delta(\overline{x})$, $j\in\mathbb{J}_0$ and $i\not\in\mathbb{J}_0$ we have
 $$ \overline{x}_j-x_j \leq |\overline{x}_j-x_j|\leq \|\overline{x}-x\|<\delta\leq\frac{1}{2}\overline{x}_j,\text{~~and~~} x_i=|x_i-\overline{x}_i|\leq \|\overline{x}-x\|<\delta \leq \frac{1}{2}\overline{x}_j.$$
 Altogether, $0<\overline{x}_j/2\leq x_j$, $x_i<\overline{x}_j/2$, and therefore $\mathcal{J}_s(x)=\{\mathbb{J}_0\}$. By Proposition~\ref{prop:projKs} $P_{\mathcal{K}_s}(x)$ is single-valued, hence $\mathcal{K}_s$ is prox-regular at $\overline{x}$.
\end{proof}

\begin{proposition}[Prox-regularity of $\mathcal{S}_s$]\label{prop:proxRegSs}
 Let $\overline{X}\in \mathcal{S}_s$ with $s\in\{1,\dots,(m-1)\}$ and $m\geq 2$. Then $\mathcal{S}_s$ is prox-regular at $\overline{X}$ if and only if $\rank(\overline{X})=s$.
\end{proposition}
\begin{proof}
 By Proposition~\ref{prop:proxRegKs}, the set $\mathcal{K}_s$ is prox-regular at $\lambda(\overline{X})$ if and only if $\|\lambda(\overline{X})\|_0=s$. The result now follows from \cite[Th.~9]{lewis2008proxRegSpectral}.
\end{proof}

\begin{remark}\label{re:proxRegSs}
 In both of the above propositions, we excluded the possibility that $s=m$. However, in this case, the sets are convex hence everywhere prox-regular.
\qede\end{remark}

\section{Implications and Applications}\label{sec:application}
In this section we give instances of problems to which the derived normal cone formulae and regularity results apply. In particular, we consider instances of the ($2$-set) \emph{feasibility problem}. That is, given sets $C_1,C_2\subseteq\E$, we aim to
 \begin{equation}\label{eq:feasibility}
  \text{find~}x\in C_1\cap C_2\neq\emptyset.
 \end{equation}
 
 The study of these regularity properties are of interest, in part, because they are precisely the ingredient required by the current state-of-art theory for non-convex \emph{projection algorithms} to guarantee local convergence \cite{restrictedNormalConesAffine,restrictedNormalCones,HLN2014sparse,phan2014linearDR}. Examples of such algorithms include \emph{cyclic projection} and \emph{Douglas--Rachford methods}. These algorithms are intimately related to a number fundamental algorithms arising the optics literature including the \emph{error reduction}, the \emph{difference-map} and \emph{hybrid input-output} algorithms \cite{bauschke2002phase,fienup,elser}.
 
 In each of the following examples, we first state the problem and then give an appropriate feasibility formulation or relaxation. The conditions required for the regularity properties studied to hold are then considered. Particular attention is given to our final example which arise in low-rank distance matrix reconstruction problems. In this direction, we given a characterization of the sufficient conditions required by \cite{phan2014linearDR} to guarantee local linear convergence of the \emph{Douglas--Rachford method}. This result complements the empirical studies \cite{BT2014inverse,ABT2014matrix} in which local linear convergence was observed. Moreover, to our knowledge, there are no other results in the literature which provide sufficient conditions for convergence of the Douglas--Rachford algorithm applied to this problem.

\subsection{Sparse and Low-Rank Linear Systems}\label{sec:linear system}
Consider the problem of finding the sparsest non-negative solution to a linear system. That is,
   \begin{equation}\label{eq:l0min nonneg}
   \min_{x\in\mathbb{R}^n}\{\|x\|_0: Ax=b,\,x\geq 0\},
  \end{equation} 
where $A\in\mathbb{R}^{m\times n}$ and $b\in\mathbb{R}^n$.
  
When a desired sparsity bound, $s$, can be given \emph{a priori}, the problem can be reformulated as
 \begin{equation}\label{eq:sparseLinear}
  \text{find~}x\in\mathbb{R}^m\text{~such that~}Ax=b,\,x\geq 0,\,\|x\|_0\leq s.
 \end{equation}
 As we have already mentioned, problems of precisely this kind arise when a signal $x$, which is known \emph{a priori} to be non-negative, purely real and sparse, is to be recovered from the under-determined linear system specified by the measurement matrix $A$ and observation vector $b$. Another example is given by \emph{mixed-integer linear programs}.

\begin{example}[Mixed-binary linear program]
	Let $M>0$ and consider a mixed-binary linear program having feasible region $F$ given by
		$$F:=\left\{(x,y)\in\mathbb{R}^m\times\{0,1\}^m:Ax=b,\,0\leq x\leq My,\,\sum_{j=1}^my_j\leq s\right\}.$$
	Finding a pair contained in $F$ can be formulated in terms of \eqref{eq:sparseLinear}. Precisely, 
 $$(x,y)\in F \iff \left\{\begin{array}{l}
                   x\in\mathbb{R}^m\text{ satisfies \eqref{eq:sparseLinear} and }             \|x\|_\infty\leq M,\\
                   y\in\{0,1\}^m\text{ with }y_j=1\text{ whenever }x_j\neq 0.
                 \end{array}\right.$$
	Regarding this equivalence, note that in order to find a pair $(x,y)\in F$, it suffices to find an $x\in\mathbb{R}^m$ satisfying \eqref{eq:sparseLinear} with $\|x\|_\infty\leq M$. Using such a point, a feasible vector $y\in\{0,1\}^m$ can easily be computed since $\|x\|_0\leq s$.
                 
	Finding feasible points of mixed-integer linear programs is important as they are frequently solved using \emph{branch-and-bound} techniques which require feasible solutions during \emph{pruning steps}.
\qede\end{example}

Returning our attention to the general problem, we observe \eqref{eq:sparseLinear} to be equivalent to the feasibly problem \eqref{eq:feasibility} with constraint sets
  \begin{equation}\label{eq:sparsity constriants vec}
   \begin{split}
     C_1 &:= \{x\in\mathbb{R}^m:Ax=b\},\\
     C_2 &:= \mathcal{K}_s=\{x\in\mathbb{R}^n:x\geq 0,\,\|x\|_0\leq s\}.
   \end{split}
  \end{equation}
  
 We may now give a characterisation of strong regularity of the collection $\{C_1,C_2\}$ using the results of previous sections. 
 \begin{proposition}
   Suppose $\overline{x}\in C_1\cap C_2$ where the constraint sets $C_1$ and $C_2$ are as defined in \eqref{eq:sparsity constriants vec}. Then $\{C_1,C_2\}$ is strongly regular at $\overline{x}$ if and only if, for all $y\in\range A^\trans\setminus\{0\}$, either $\overline{x}\odot y=0$, or (b) $y\not\geq 0$ and $\|y\|_0>m-s$.
 \end{proposition}    
 \begin{proof}
   Suppose $\overline{x}\in C_1\cap C_2$. The normal cone to $C_1$ at $\overline{x}$ is given by (see \cite[Sec.~2.1,~Exer.~4]{ConvAnalNonOpt})
   $$N_{C_1}(\overline{x}) = \range A^\trans.$$
  The result now follows from Proposition~\ref{prop:strongRegKs} with $\Omega:=C_1$.
 \end{proof}

 We now turn our attention to the symmetric matrix analogue. Consider the problem of finding a minimum rank positive semi-definite matrix solution to a linear system. That is,
  \begin{equation}\label{eq:lowRankSDPminimisation}
   \min_{X\in\Sym^n}\{\rank X:\left(\langle A_j,X\rangle\right)_{j=1}^p=b,\,X\succeq 0\}.
  \end{equation}  
 As before, when a desired bound on the rank, $s$, can be given \emph{a priori}, the problem can be reformulated as
  \begin{equation}\label{eq:lowRankSDP}
   \text{find~}X\in\Sym^m\text{~such that~}\left(\langle A_j,X\rangle\right)_{j=1}^p=b,\,X\succeq 0,\,\rank(X)\leq s,
  \end{equation}
 where $A_j\in\Sym^m$ for each $j\in\{1,\dots,p\}$ and $b\in\mathbb{R}^p$.

 Problems of this kind arise in rank-constrained semi-definite programming \cite[Exer.~2.4]{ApproxAlgSemi}, an NP-hard problem.

 We observe \eqref{eq:lowRankSDP} to be equivalent to the two set feasibility problem \eqref{eq:feasibility} with constraint sets
  \begin{equation}\label{eq:low rank psd matrix}
   \begin{split}
    C_1 &:=\{X\in\Sym^m:\left(\langle A_j,X\rangle\right)_{j=1}^p=b\}, \\    
    C_2 &:=\mathcal{S}_s=\{X\in\Sym^m:X\succeq 0,\,\rank X\leq s\}.
   \end{split}
  \end{equation}
  
   We may now give a characterization of strong regularity of the collection $\{C_1,C_2\}$.
 \begin{proposition}
   Suppose $\overline{X}\in C_1\cap C_2$ where the constraint sets $C_1$ and $C_2$ are as defined in \eqref{eq:low rank psd matrix}. Then $\{C_1,C_2\}$ is strongly regular at $\overline{X}$ if and only if, for all $Y\in\vspan\{A_1,A_2,\dots,A_p\}\setminus\{0\}$, either (a) $\overline{X}Y=0$, or (b) $Y\not\succeq 0$ and $\rank (Y)>m-s$.
 \end{proposition}
 \begin{proof}
 Suppose $\overline{X}\in C_1\cap C_2$. Then the normal cone to $C_1$ at $\overline{X}$ is given by (see \cite[Sec~2.1,~Exer~4]{ConvAnalNonOpt} \& \cite[Sec.~5]{vn93})
  $$N_{C_1}(\overline{X})=\range \left(\left(\langle A_j,\cdot \rangle\right)_{j=1}^p\right)^\ast=\vspan\{A_1,A_2,\dots,A_p\}.$$
 The result now follows from \ref{prop:strongRegSs} with $\Omega:=C_1$.
\end{proof}

\begin{remark}[Sparsity upper bounds]
 In order to apply Formulations~\eqref{eq:sparseLinear} and \eqref{eq:lowRankSDP}, it is necessary that the sparsity/rank parameter $s$ is know \emph{a priori}. While this is not always possible, we emphasise that it is usually not the case that the \emph{optimal} sparsity parameter (in the sense of \eqref{eq:l0min nonneg} or \eqref{eq:lowRankSDPminimisation}) is needed. When projection or reflection methods are applied to a formulation only prescribing an upper bounded for the parameter, there is nothing to prevent sparser/lower-rank solutions from being returned. For further details, see \cite{HLN2014sparse}.
\end{remark}

\subsection{Low-Rank Euclidean Distance Matrix Completion}
 We now use the derived normal cone formulae to provide theoretical justification for an application of the  \emph{Douglas--Rachford method}. Given constraint sets $C_1$ and $C_2$ with non-empty intersection, and initial point $x_0$, this method generates a sequence $(x_n)$ by selecting
  $$x_{n+1}\in Tx_n:=\left(\frac{I+R_{C_2}R_{C_1}}{2}\right)x_n.$$
 Here $R_{\Omega}:=2P_{\Omega}-I$ denotes the \emph{reflection mapping w.r.t.~$\Omega$}. If the sequence $(x_n)$ converges to a fixed point $\overline{x}\in\Fix T:=\{x:x\in Tx\}$, then there is an element $\overline{p}\in P_{C_1}(\overline{x})$ such that $\overline{p}\in C_1\cap C_2$. That is, the point $\overline{p}$ solves the \emph{feasibility problem} (rather than the fixed point $\overline{x}$ itself). For closed convex sets, this method is well-understood and global convergence guaranteed \cite{monotoneOperator}.

Recall a \emph{Euclidean distance matrix (EDM)} is a matrix $D=(D_{ij})\in\Sym^{m+1}$ such that there exists points $p_1,p_2,\dots,p_m,p_{m+1}\in\mathbb{R}^q$ such that
 \begin{equation}\label{eq:EDM}
  D_{ij}=\|p_i-p_j\|^2\text{~for all~}(i,j)\in\{1,2,\dots,(m+1)\}\times\{1,2,\dots,(m+1)\}.
 \end{equation}
Clearly any EDM is non-negative and hollow ({\em i.e.,} contains only zeros along its main diagonal). When \eqref{eq:EDM} holds for a set points in $\mathbb{R}^q$, we say $D$ is \emph{embeddable} in $\mathbb{R}^q$. If $D$ is embeddable in $\mathbb{R}^q$ but not $\mathbb{R}^{q-1}$ then we say $D$ is \emph{irreducibly embeddable} in $\mathbb{R}^q$.

In \cite{ABT2014matrix,BT2014inverse} the authors consider the problem of reconstructing a \emph{low-rank Euclidean distance matrix}, $X$, knowing only a subset of its entries. Let $D=(D_{ij})$ be the partial EDM with the position of the known entries specified by the index set $\mathcal{I}$. Assuming the EDM is \emph{embeddable} in $\mathbb{R}^s$, this problem can be formulated as the feasibility problem \eqref{eq:feasibility} with constraint sets 
\begin{equation}\label{eq:lowrankEDM}\begin{split}
 C_1 &= \left\{X\in\Sym^{m+1}:X\geq 0,\,X_{ij}=D_{ij}\text{ for all }(i,j)\in\mathcal{I}\right\},\\
 C_2 &= \left\{X\in\Sym^{m+1}:Q(-X)Q=\begin{bmatrix}
                                              \widehat{X} & d \\
                                              d^\trans         & \delta \\
                                            \end{bmatrix},\,\widehat{X}\in\mathcal{S}_s,\,d\in\mathbb{R}^{m},\,\delta\in\mathbb{R}\right\}.
\end{split}\end{equation}
Here the matrix $Q\in\Orth^{m+1}$ is the \emph{Householder matrix} given by
   \begin{equation*}
    Q=I-\frac{2vv^\trans}{v^\trans v},\text{ where }v=(1,\,1,\,\dots,\, 1,\,1+\sqrt{m+1})^\trans\in\mathbb{R}^{m+1}.
   \end{equation*}   
We denote by $G:\Sym^{m+1}\to\Sym^{m+1}$ the linear isometry $X\mapsto Q(-X)Q$. Further observe that $G^2=I$. This feasibility problems is a consequence of the following EDM characterization. For further details, the reader is referred to \cite{ABT2014matrix,BT2014inverse,haydenWells,qi2014computing}.

\begin{theorem}[EDM characterization {\cite[Th.~3.3]{haydenWells}}]\label{th:EDMchar}
	Let ${X\in\Sym^{m+1}}$ be a  non-negative, hollow matrix and denote by $\widehat{X}\in\Sym^m$ the upper-left block of the matrix $Q(-X)Q$ (see \eqref{eq:lowrankEDM}). Then $X$ is a Euclidean distance matrix if and only if $\widehat{X}\in\Sym^m_+$. Moreover, when $X$ is an EDM, it is irreducibly embeddable in $\mathbb{R}^s$ where $s=\rank(\widehat{X})\leq m$.
\end{theorem}

\begin{remark}[Properties of the constraint sets]\label{re:constraints}
 Observe that the constraint set $C_1$ is closed and convex. In fact, it is the intersection of a closed convex cone and a closed affine subspace. The constraint set $C_2$ is non-convex and can alternatively be described as the pre-image of the set $\mathcal{S}_s$ under the linear operator which maps $X\in\Sym^{m+1}$ to $\widehat{X}\in\Sym^m$ as given in \eqref{eq:lowrankEDM}. Note that both the mathematical challenging and physically meaning cases arise when the sparsity bound satisfies $1<s<m$.
\qede\end{remark}

The goal of the remainder of this section is to provide conditions for which the following local convergence result due to Phan \cite{phan2014linearDR} applies. To apply the result we shall need to check two properties: \emph{super-regularity} and \emph{strong regularity}.

\begin{theorem}[Local convergence {\cite[Th.~4.3]{phan2014linearDR}}]\label{th:phan}
 Suppose $\Omega_1$ and $\Omega_2$ are super-regular at $\overline{x}\in \Omega_1\cap \Omega_2$ and $\{\Omega_1,\Omega_2\}$ is strongly regular at $\overline{x}$. For any $x_0$ sufficiently close to $\overline{x}$, the Douglas--Rachford method converges to a point $\Omega_1\cap \Omega_2$ with $R$-linear rate.
\end{theorem}
 
 Rather than verifying super-regularity directly, we instead consider the case in which the constraints in \eqref{eq:lowrankEDM} satisfy a stronger property. Later we shall see that for the application considered, this is in fact the case. We also note that, in light of Remark~\ref{re:affine-hull regularity}, there is nothing to gain in applying \cite[Th.~4.7]{phan2014linearDR} (the affine-hull regularity analogue of Theorem~\ref{th:phan}).
 
 \begin{proposition}[Prox-regularity of constraints]\label{prop:proxReg}
   Suppose $\overline{X}\in C_1\cap C_2$ with $\rank\widehat{X}=s$ where $C_1,C_2$ and $\widehat{X}$ are as defined in \eqref{eq:lowrankEDM}. Then $C_1$ and $C_2$ are prox-regular (and hence also super-regular) at $\overline{X}$.
 \end{proposition}
\begin{proof}
	Since $\overline{X}\in C_1\cap C_2$, both $C_1$ and $C_2$ are nonempty. As noted in Remark~\ref{re:constraints}, $C_1$ is closed and convex, and hence is everywhere prox-regular. 
	To deduce the prox-regularity of $C_2$, first observe that, by Proposition~\ref{prop:proxRegSs} and Remark~\ref{re:proxRegSs}, the set $\mathcal{S}_s$ is prox-regular at $\widehat{X}$, and hence so too is the set
 $$ G(C_2) = \left\{\begin{bmatrix}
				  	\widehat{Z} & d \\
				  	d^\trans         & \delta \\
			  \end{bmatrix} : \widehat{Z}\in\mathcal{S}_s,\,d\in\mathbb{R}^{m},\,\delta\in\mathbb{R}\right\}$$
	is prox-regular at $G(\overline{X})$. 						  
	Since $G$ is an isometry and $G^2=I$, we deduce that the set $C_2=G(G(C_2))$ is prox-regular at the point $\overline{X}=G(G(\overline{X}))$. The fact that prox-regularity implies super-regularity can be found in \cite[Prop.~4.9]{lewis2009local}.
\end{proof}

We now turn our attention to strong regularity of $\{C_1,C_2\}$ at $\overline{X}$. The next two Propositions give the respective normal cones.

\begin{proposition}[Normal cone to $C_1$]\label{prop:NC1}
 Let $\overline{X}\in C_1$. Then
  $$N_{C_1}(\overline{X})=\left\{Y+Z:\overline{X}\odot Z=0,\,Z\leq 0,\,Y_{ij}=0\text{ for all }(i,j)\not\in\mathcal{I}\right\}.$$
 In particular, if $\overline{X}$ has zeros on the main diagonal, has strictly positive entries elsewhere, and $(j,j)\in\mathcal{I}$ for $j\in\{1,2,\dots,m\}$ then
  $$N_{C_1}(\overline{X})=\left\{Y:Y_{ij}=0\text{ for all }(i,j)\not\in\mathcal{I}\right\}.$$
\end{proposition}
\begin{proof}
 Observe that $C_1$ is the intersection of the polyhedral sets
  $$ A:=\left\{X\in\Sym^{m+1}:X_{ij}=D_{ij}\text{ for all }(i,j)\in\mathcal{I}\right\},\quad
     K:=\left\{X\in\Sym^{m+1}:X\geq 0\right\}.$$
 We note that the intersection of these sets is non-empty since $C_1$ is assumed nonempty. The \emph{(polyhedral) sum rule} \cite[Cor.~5.1.9]{ConvAnalNonOpt} thus ensures that
	 $$N_{C_1}(\overline{X})=N_A(\overline{X})+N_{K}(\overline{X}).$$     
 The normal cones to $A$ and $K$ are given by
  \begin{align*}
    N_A(\overline{X}) &= \{Y\in\Sym^{m+1}:Y_{ij}=0\text{ for all }(i,j)\not\in\mathcal{I}\},\\
    N_{K}(\overline{X}) &= \{Z\in\Sym^{m+1}:\overline{X}\odot Z=0,\,Z\leq 0\},
  \end{align*}
 and the first claim follows.
 
 In particular, suppose  $\overline{X}$ has zeros on the main diagonal and strictly positive entries elsewhere, and that $(j,j)\in\mathcal{I}$ for $j\in\{1,2,\dots,m\}$. Let $Z\in N_{K}(\overline{X})$. Since $\overline{X}\odot Z=0$, it follows that $Z$ can be non-zero only on its main diagonal. In this case, we therefore have the inclusion
   $$N_{K}(\overline{X}) \subseteq  \{Z\in\Sym^{m+1}:Z\leq 0,\,Z_{ij}=0\text{ for all }(i,j)\not\in\mathcal{I}\}\subseteq N_A(\overline{X}).$$
 As before, the result follows by an application of the (polyhedral) sum rule.
\end{proof}

\begin{proposition}[Normal cone to $C_2$]\label{prop:NC2}
 Let $\overline{X}\in C_2$ and
 $\widehat{X}\in\mathcal{S}_s$. Then
   $$N_{G(C_2)}\left(G(\overline{X})\equiv\begin{bmatrix}
                     \widehat{X} & d \\
                     d^\trans         & \delta \\
                     \end{bmatrix}\right)
             =\left\{\begin{bmatrix}
                                   \widehat{Y} & 0 \\
                                     0         & 0 \\
                                  \end{bmatrix}:
                                  \widehat{Y}\in N_{\mathcal{S}_s}(\widehat{X})\right\}.$$
 In particular, if $\rank\widehat{X}=s$ then
     $$N_{G(C_2)}\left(G(\overline{X})\equiv\begin{bmatrix}
                     \widehat{X} & d \\
                     d^\trans         & \delta \\
                     \end{bmatrix}\right)
             =\left\{\begin{bmatrix}
                                   \widehat{Y} & 0 \\
                                     0         & 0 \\
                                  \end{bmatrix}:
                                 \widehat{X}\widehat{Y}=0\right\}.$$
\end{proposition}
\begin{proof}
 The first formula follows from Theorem~\ref{th:NSs} and the fact that the normal cone of a Cartesian product is the Cartesian product of the normal cones \cite[Prop.~1.2]{mordukhovich2006generalized}. The second formula follows from Proposition~\ref{prop:NRs}.
\end{proof}

The following proposition gives a formulation of strong regularity in terms of the normal cone formula in Proposition~\ref{prop:NC1} \& \ref{prop:NC2}. That is, in terms of known objects.
\begin{lemma}[Strong regularity of $\{C_1,C_2\}$]\label{lem:strongRegEDM}
 Let $\overline{X}\in C_1\cap C_2$. Then $\{C_1,C_2\}$ is strong regularity at $\overline{X}$ if and only if 
    $$G(N_{C_1}(\overline{X}))\cap -N_{G(C_2)}(G(\overline{X}))=\{0\}.$$
\end{lemma}
\begin{proof}
 Since $G$ is self-inverse (hence self-adjoint) and linear, applying \cite[Th.~1.17]{mordukhovich2006generalized} yields
  $$N_{C_2}(\overline{X}) = G\left(N_{G(C_2)}(G(\overline{X}))\right) \implies G\left(N_{C_2}(\overline{X})\right) = N_{G(C_2)}(G(\overline{X})).$$
Hence
 \begin{align*}
  \{0\}=N_{C_1}(\overline{X})\cap -N_{C_2}(\overline{X})
  &\iff \{0\} = G(N_{C_1}(\overline{X}))\cap -G(N_{C_2}(\overline{X})) \\
  &\iff \{0\} = G(N_{C_1}(\overline{X}))\cap -N_{G(C_2)}(G(\overline{X})).
 \end{align*}
 The proof is complete.
\end{proof}

The following results shows that strong regularity of $\{C_1,C_2\}$ can be expressed in terms of a conditions which could, in principle,  be checked once a solution is known. 

\begin{theorem}[Regularity and local convergence]\label{cor:EDMworks}
Suppose $\overline{X}\in C_1\cap C_2$ is irreducibly embeddable in $\mathbb{R}^s$, has zeros on the main diagonal and has strictly positive entries elsewhere. Then both $C_1$ and $C_2$ are prox-regular at $\overline{X}$. The collection $\{C_1,C_2\}$ is strongly regular at $\overline{X}$ if and only if there exists no non-zero pair $(Y,\widehat{Y})\in\Sym^{m+1}\times\Sym^m$ such that
 \begin{equation}\label{eq:EDMworks}
  G(Y)=\begin{bmatrix}
                     \widehat{Y} & 0 \\
                     0           & 0 \\
                     \end{bmatrix},
                \widehat{X}\widehat{Y}=0,
                 Y_{ij}=0\text{ for all }(i,j)\not\in\mathcal{I}.\end{equation}
 Consequently, whenever \eqref{eq:EDMworks} holds, the Douglas--Rachford method converges with $R$-linear rate whenever the initial point is sufficiently close to $\overline{X}$.
\end{theorem}
\begin{proof}
 Prox-regularity follows from the irreducible embeddability assumption combined with Theorem~\ref{th:EDMchar} and Proposition~\ref{prop:proxReg}. Equation~\eqref{eq:EDMworks} follows immediately by combining the normal cone formulae in Propositions~\ref{prop:NC1}~\&~\ref{prop:NC2} with the equivalent definition of strong regularity in Lemma~\ref{lem:strongRegEDM}.  
 Local convergence of the Douglas--Rachford algorithm now follows from Theorem~\ref{th:phan}.
\end{proof}

\section{Conclusion}
 In this paper, novel formulae for the Mordukhovich normal cones to the sets of non-negative sparse vector and low-rank positive semi-definite matrix have been provided. These normal cones are of interest, for instance, in examining the regularity properties of constraint sets in various applications. As a concrete application of our results, the precise conditions under which the current non-convex state-of-the-art convergence theory for the Douglas--Rachford algorithm holds, when applied to low-rank Euclidean distance matrix reconstruction, have been characterized. Regarding this particular application, an avenue for further investigation would be to numerically check condition \eqref{eq:EDMworks}, and thus determine if strong regularity holds in real-world datasets. It is worth noting that, in \eqref{eq:EDMworks}, the matrix $\widehat{X}$ and index set $\mathcal{I}$ are specific to each problem instance. In practice, these are determined from collected data and will vary from experiment to experiment. For this reasons, it would there also be useful to see if a probabilistic argument can be used to show that the strong regularity condition in Theorem~\ref{cor:EDMworks} holds is some \emph{generic} sense.

\paragraph*{Acknowledgements.}
The author would like to thank Jonathan Borwein for his suggestions. The author is supported by the Deutsche Forschungsgemeinschaft Research Training Grant 2088. The work was partly performed during the author's candidature at the University of Newcastle with the support of an Australian Postgraduate award.

\end{document}